\documentclass[11pt]{article}
\usepackage{amsmath,amsfonts,amssymb,amsthm,enumerate,graphicx}
\usepackage{tikz,authblk}
\usepackage{subfig}
\usepackage{hyperref}

\newtheorem{theorem}{{\bf Theorem}}[section]
\newtheorem{conjecture}[theorem]{{\bf Conjecture}}
\newtheorem{corollary}[theorem]{{\bf Corollary}}
\newtheorem{definition}[theorem]{{\bf Definition}}
\newtheorem{example}[theorem]{{\bf Example}}
\newtheorem{lemma}[theorem]{{\bf Lemma}}
\newtheorem{notation}[theorem]{{\bf Notation}}
\newtheorem{proposition}[theorem]{{\bf Proposition}}
\newtheorem{remark}[theorem]{{\bf Remark}}
\newtheorem{prob}[theorem]{{\bf Problem}}

\newcommand{\FF}{ \ensuremath{\mathbb{F}}}
\newcommand{\QQ}{ \ensuremath{\mathbb{Q}}}
\newcommand{\ZZ}{ \ensuremath{\mathbb{Z}}}

\newcommand{\TPSS}{S^{\hspace{.2mm}2}\! \times \hspace{-3.3mm}_{-} \,
S^{\hspace{.1mm}1}}

\begin{document}

\author[a] {Bhaskar Bagchi}
\author[b] {Basudeb Datta}
\author[c] {Jonathan Spreer}

 \affil[a] {Theoretical Statistics and Mathematics Unit, Indian Statistical Institute, Bangalore 560\,059,
 India. bbagchi@isibang.ac.in}
 \affil[b] {Department of Mathematics, Indian Institute of Science, Bangalore 560\,012, India.
 dattab@math.iisc.ernet.in.}
 \affil[c] {School of Mathematics and Physics, The University of Queensland, Brisbane QLD 4072, Australia.
 j.spreer@uq.edu.au.}

\title{Tight triangulations of closed 3-manifolds}


\date{{\bf (European Journal of Combinatorics 54 (2016) 103--120.)}}

\maketitle

\vspace{-10mm}

\begin{abstract}
A triangulation of a closed 2-manifold is tight with respect to a field of characteristic
two if and only if it is neighbourly; and it is tight with respect to a field of odd characteristic if and only
if it is neighbourly and orientable. No such characterization of tightness was previously known for higher
dimensional manifolds. In this paper, we prove that a triangulation of a closed 3-manifold is tight with respect
to a field of odd characteristic if and only if it is neighbourly, orientable and stacked. In consequence, the
K\"{u}hnel-Lutz conjecture is valid in dimension three for fields of odd characteristic.

Next let $\FF$ be a field of characteristic two. It is known that, in this case, any neighbourly and stacked
triangulation of a closed 3-manifold is $\FF$-tight. For closed, triangulated  3-manifolds with at most 71
vertices or with first Betti number at most 188, we show that the converse is true. But the possibility of the
existence of an $\FF$-tight, non-stacked triangulation on a larger number of vertices remains open. We prove the
following upper bound theorem on such triangulations. If an $\FF$-tight triangulation of a closed 3-manifold has
$n$ vertices and first Betti number $\beta_1$, then $(n-4)(617n- 3861) \leq 15444\beta_1$. Equality holds here if
and only if all the vertex links of the triangulation are connected sums of boundary complexes of icosahedra.
\end{abstract}

\noindent {\small {\em MSC 2010\,:} 57Q15, 57R05.

}

\section{Introduction}

All simplicial complexes considered in this paper are finite and abstract. All homologies are simplicial
homologies with coefficients in a field $\FF$. The vertex set of a simplicial complex $X$ will be denoted by
$V(X)$. For $A \subseteq V(X)$, the induced subcomplex $X[A]$ of $X$ on the vertex set $A$ is defined by $X[A] :=
\{\alpha\in X \, : \, \alpha\subseteq A\}$. A simplicial complex $X$ is said to be a {\em triangulated manifold}
if it triangulates a manifold, i.e., if the geometric carrier $|X|$ of $X$ is a topological manifold. A closed,
triangulated  $d$-manifold $X$ is said to be {\em $\FF$-orientable} if $H_d(X; \FF) \neq 0$. So, for a field
$\FF$ of characteristic two, any closed, triangulated  manifold is $\FF$-orientable.

Taking his cue from pre-existing notions of tightness in the theory of smooth and polyhedral embedding of
manifolds in Euclidean spaces, K\"{u}hnel \cite{Ku} introduced the following precise notion of tightness of a
simplicial complex with respect to a field.

\begin{definition} \label{def:tight}
{\rm Let $X$ be a simplicial complex and $\FF$ be a field. We say that $X$ is} tight with respect to $\FF$ {\rm
(in short,} $\FF$-tight{\rm ) if (a) $X$ is connected, and (b) the $\FF$-linear map $H_{\ast}(Y; \FF) \to
H_{\ast}(X; \FF)$, induced by the inclusion map $Y \hookrightarrow X$, is injective for every induced subcomplex
$Y$ of $X$. }
\end{definition}

Recall that, if $X$ is a simplicial complex of dimension $d$, then its {\em face numbers} $f_i(X)$ are  defined
by $f_i(X) := \#\{\alpha\in X \, : \, \dim(\alpha)=i\}, \,  0\leq i\leq d$. For $k\geq 2$, a simplicial complex
$X$ is said to be {\em $k$-neighbourly} if any set of $k$ vertices of $X$ forms a face, i.e., if $f_{k-1}(X) =
\binom{f_0(X)}{k}$. A 2-neighbourly simplicial complex is called {\em neighbourly}.

\begin{definition} \label{def:sminimal}
{\rm A simplicial complex $X$ is said to be} strongly minimal {\rm if, for every triangulation $Y$ of the
geometric carrier $|X|$ of $X$, we have $f_i(X) \leq f_i(Y)$ for all $i$, $0\leq i\leq \dim(X)$. }
\end{definition}

Thus, a strongly minimal triangulation of a topological space, if it exists, is the most economical among all
possible triangulations of the space. Unfortunately, there are very few combinatorial criteria available in the
literature which ensure strong minimality. The notion of tightness is of great importance in combinatorial
topology because of the following tantalizing conjecture \cite{KL}.

\begin{conjecture}[K\"{u}hnel-Lutz] \label{conj:KL}
For any field $\FF$, every $\FF$-tight, closed, triangulated manifold is strongly minimal.
\end{conjecture}

Intuitively, $\FF$-tightness of a triangulated manifold $X$ means that all parts of $X$ are essential in order to
capture the $\FF$-homology of the topological space $|X|$. In view of this intuition, Conjecture \ref{conj:KL}
appears to be entirely plausible. However, Example \ref{exam:nsmtight} shows that this intuition is not
correct for arbitrary simplicial complexes.

\medskip

Before we proceed, let us recall the notion of stackedness from \cite{MNStacked}.

\begin{definition} \label{def:stacked}
{\rm A triangulated manifold $\Delta$ of dimension $d+1$ is said to be {\em stacked} if all its faces of
codimension (at least) two are contained in the boundary $\partial\Delta$. A closed, triangulated manifold $M$ of
dimension $d$ is said to be {\em stacked} if there is a stacked triangulated manifold $\Delta$ of dimension $d+1$
such that $M = \partial \Delta$. }
\end{definition}

In particular, a {\em stacked sphere} is a triangulated sphere which may be realized as the boundary of a stacked
triangulated ball.

\begin{definition} \label{def:lstacked}
{\rm A triangulated manifold is said to be {\em locally stacked} if all its vertex links are stacked spheres or
stacked balls.}
\end{definition}

Clearly, all stacked triangulated manifolds are locally stacked, but the converse is false (see Example
\ref{exam:Lutz}). Locally stacked triangulations are a rich source of tight triangulations. With just a few
exceptions, all known tight triangulated manifolds are locally stacked. From \cite{BDStacked}, we know that
when $d\neq 3$, an $\FF$-orientable, neighbourly, locally stacked, triangulated $d$-manifold is $\FF$-tight. See Effenberger \cite{ef} for a proof of this in the case $\FF = \ZZ_2$. But this result is not true in dimension 3 (c.f. Example
\ref{exam:Lutz}). Due to \cite[Theorem 2.24; case $k=1$]{BDStacked} we have the following.

\begin{proposition}[Bagchi-Datta] \label{prop:BDTh2.24}
Let $M$ be a locally stacked, $\FF$-orientable, neighbourly, closed, triangulated $3$-manifold. Then the
following are equivalent
\begin{enumerate}[{\rm (i)}]
\item $M$ is $\FF$-tight,

\item $M$ is stacked, and

\item $\binom{f_0(M)-4}{2} = 10 \beta_1(M; \FF)$.
\end{enumerate}
\end{proposition}

Thus, all $\FF$-orientable, neighbourly, stacked, closed, triangulated 3-manifolds are $\FF$-tight (for any field
$\FF$). Note that for $d\geq 4$, the notions of stackedness and local stackedness coincide for triangulated
manifolds (\cite{Ka, BDStacked, MNStacked}). In conjunction with Proposition \ref{prop:BDTh2.24}, this shows that
it is stackedness (as opposed to local stackedness) which seems to be central to tightness. The notion of
stackedness was introduced by Walkup \cite{Wa} and McMullen \& Walkup \cite{MW} in the context of triangulated
spheres. The close relationship between stackedness and tightness has been highlighted in \cite{BaTight, BDStacked, BDStellated, 
Mu}. This is further borne out by Theorem 4.8 of this article.

In \cite{BDStellated}, it was shown that the K\"{u}hnel-Lutz conjecture is true for locally stacked manifolds.
Namely, we have

\begin{proposition}[Bagchi-Datta] \label{prop:sminimal}
For any field $\FF$, every $\FF$-tight, locally stacked, closed, triangulated manifold is strongly minimal.
\end{proposition}

The first main result of this paper (in Section 4) is an extension of Proposition \ref{prop:BDTh2.24}: If a closed,
triangulated 3-manifold is tight with respect to a field of odd characteristic then it must be (orientable,
neighbourly and) stacked. This result answers Question 4.5 of \cite{DM} affirmatively, in the case of odd
characteristic. As a consequence of Proposition \ref{prop:sminimal} it follows that the K\"{u}hnel-Lutz
conjecture is true in a special case, namely, if ${\rm char}(\FF) \neq 2$ then any $\FF$-tight, closed,
triangulated 3-manifold is strongly minimal.

Let $X_1$ and $X_2$ be two triangulated $d$-manifolds intersecting in a common facet ($d$-face) $\alpha$. That
is, $X_1 \cap X_2 = \bar{\alpha}$, where $\bar{\alpha}$ denotes $\alpha$ together with all of its subfaces. Then
$X_1\#X_2 = (X_1 \cup X_2 ) \setminus \{\alpha\}$ is said to be the {\em connected sum of $X_1$ and $X_2$ along
$\alpha$} (for a more general definition, see the end of Section 2). Let $I = I^2_{12}$ be the boundary complex
of the icosahedron. Thus, $I$ is a triangulated 2-sphere on 12 vertices. It is well known that $I$ is the unique
triangulation of $S^2$ in which each vertex is of degree 5. We introduce:

\begin{definition} \label{def:icosian}
{\rm A triangulated $2$-sphere is said to be {\em icosian} if it is a connected sum of finitely many copies of
$I$. A triangulated $3$-manifold is said to be {\em locally icosian} if all its vertex links are icosian.}
\end{definition}

In \cite{SpOdd}, the third author proved the following interesting upper bound theorem for tight triangulations
of odd dimensional manifolds.

\begin{proposition}[Spreer] \label{prop:Spreer}
Let $M$ be an $(\ell- 1)$-connected, closed, triangulated $(2\ell + 1)$-manifold. If $M$ is $\FF$-tight, then
${\binom{\lfloor\frac{f_0(M)}{2}\rfloor-1}{\ell+1}\binom{\lceil
\frac{f_0(M)}{2}\rceil-1}{\ell+1}}/{\binom{f_0(M)-1}{\ell+1}}\leq \beta_{\ell}(M; \FF)$.
\end{proposition}

In Section 5, we consider fields of characteristic 2. According to Proposition \ref{prop:BDTh2.24}, every
neighbourly, stacked, closed, triangulated 3-manifold $M$ is $\ZZ_2$-tight. We do not know whether the converse
is true or not. But in this paper we prove that if $M$ is a $\ZZ_2$-tight, closed, triangulated 3-manifold with
$f_0(M) \leq 71$ or $\beta_1(M; \ZZ_2) \leq 188$, then $M$ must be stacked and neighbourly (and therefore
$\binom{f_0(M)-4}{2} = 10\beta_1(M; \ZZ_2)$). We also show that, in general, each vertex link of a $\ZZ_2$-tight,
closed, triangulated 3-manifold must be a connected sum of $I$'s and $S^2_4$'s. Further, we prove that any
$\ZZ_2$-tight, closed, triangulated 3-manifold $M$ satisfies the following upper bound on $f_0(M)$:
\begin{align}
(f_0(M)-4)(617f_0(M)-3861) \leq 15444\beta_1(M; \ZZ_2). \nonumber
\end{align}
Equality holds here if and only if $M$ is locally icosian. In conjunction with the fact that $\binom{f_0(M)-4}{2}
= 10\beta_1(M; \ZZ_2)$ when $f_0(M) \leq 71$, this inequality improves upon the upper bound of Proposition
\ref{prop:Spreer} in case $l=1$. We also prove that, if there is a non-stacked, $\FF$-tight, triangulated
3-manifold $M$, then its integral homology group $H_1(M;\ZZ)$ must have an element of order 2. The results of
Section 5 were largely suggested by extensive machine computations using simpcomp \cite{ES}. Altogether, these
results impose severe restrictions on the topology of $3$-manifolds admitting tight triangulations (cf. Corollary
\ref{coro:5.14}).

In Section 6, we present some examples to show that converses and generalizations of several results proved here
are not true.

\begin{remark} \label{remark:1.10}
{\rm Tight triangulations of $3$-manifolds are not as rare as one might think. Recently, it has been found in
\cite{BDSS2} that $(20m+9)$-vertex tight, triangulated $3$-manifolds exist for all $m\leq 5$. The paper
\cite{BDSS2} lists 76 non-isomorphic tight triangulations with these parameters including Walkup's $9$-vertex
$3$-dimensional Klein bottle. Existence of these examples makes it natural to pose the following conjecture. }
\end{remark}

\begin{conjecture} \label{conj:1.11}
There exist $n$-vertex, tight, closed, triangulated $3$-manifolds for all $n \equiv 9$ $($mod $20)$.
\end{conjecture}

However, all the examples in \cite{BDSS2} are actually stacked. Indeed, in view of Theorems \ref{theorem:5.6} and
\ref{theorem:5.11} of this paper, if non-stacked, tight, triangulated 3-manifolds are to exist, then their
structure and parameters must be extremely restrictive. These results suggest the following conjecture (which we
prove here in odd characteristic).

\begin{conjecture} \label{conj:1.12}
A closed, triangulated $3$-manifold is $\FF$-tight $($if and\,$)$ only if it is $\FF$-orientable, neighbourly and
stacked.
\end{conjecture}

In the Oberwolfach workshop on `Geometric and Algebraic Combinatorics' held in February, 2015, the first author
gave a lecture based on this article. There he posed the following \cite{MFO}. 

\begin{prob}[Bagchi] \label{prob:Ba}
Let $X$ be an $\FF$-tight simplicial complex such that the link in $X$ of some vertex is an $\FF$-homology
$(d-1)$-sphere. Then prove or disprove: $X$ must be a triangulation of a closed $\FF$-homology $d$-manifold. 
\end{prob}

Notice that this is trivial if $d=1$. Theorem \ref{theorem:3.5} provides an affirmative solution of this
problem for $d=2$.

\section{Preliminaries on stacked and tight triangulations}

In this section we recapitulate a few elementary consequences of tightness. For completeness, we include their
proofs. We shall use\,:

\begin{notation} \label{notation1}
{\rm If $x$ is a vertex of a simplicial complex $X$, then $X^x$ and $X_x$ will denote the antistar and the link
(respectively) of $x$ in $X$. Thus,
\begin{align}
X^x & := \{\alpha\in X \, : \, x \not\in\alpha\} = X[V(X)\setminus\{x\}], \nonumber \\
X_x & := \{\alpha\in X : x\not\in\alpha, \alpha\sqcup\{x\}\in X\}. \nonumber
\end{align}
}
\end{notation}

A face $\{u_1, \dots, u_m\}$ in a simplicial complex is also denoted by $u_1u_2\cdots u_m$. If $X$ is a
simplicial complex and $a\not\in X$ is an element then the {\em cone} with apex $a$ and base $X$ is the
simplicial complex $X \cup \{\alpha\cup\{a\} \, : \, \alpha\in X\}$ and is denoted by $a\ast X$. For a simplicial
complex $X$ of dimension $d$, and for $0\leq k\leq d$, the {\em $k$-skeleton} ${\rm skel}_k(X)$ is defined by
\begin{align}
{\rm skel}_k(X) & := \{\alpha\in X \, : \, \dim(\alpha) \leq k\}. \nonumber
\end{align}

\begin{lemma} \label{lemma:2.2}
Every $\FF$-tight simplicial complex is neighbourly.
\end{lemma}

\begin{proof}
Suppose, if possible, $x\neq y$ are two vertices of an $\FF$-tight simplicial complex $X$ such that $xy$ is not
an edge of $X$. Let $Y$ be the induced subcomplex of $X$ on the set $\{x, y\}$. Then $\beta_0(Y; \FF) = 2 > 1 =
\beta_0(X; \FF)$, so that $H_0(Y; \FF) \to H_0(X; \FF)$ can not be injective. This is a contradiction since $X$
is $\FF$-tight.
\end{proof}

\begin{lemma} \label{lemma:2.3}
Every induced subcomplex of an $\FF$-tight simplicial complex is $\FF$-tight.
\end{lemma}

\begin{proof}
Let $Y$ be an induced subcomplex of an $\FF$-tight simplicial complex $X$. By Lemma \ref{lemma:2.2}, $X$ is
neighbourly and hence $Y$ is also neighbourly. So $Y$ is connected. Let $Z$ be an induced subcomplex of $Y$. Then
$Z$ is an induced subcomplex of $X$ also. Since the composition of the linear maps $H_{\ast}(Z; \FF) \to
H_{\ast}(Y; \FF) \to H_{\ast}(X; \FF)$ is injective, the first of them must be injective. Thus, $H_{\ast}(Z; \FF)
\to H_{\ast}(Y; \FF)$ is injective for all induced subcomplexes $Z$ of $Y$. Therefore, $Y$ is $\FF$-tight.
\end{proof}

\begin{lemma} \label{lemma:2.4}
If $X$ is an $\FF$-tight simplicial complex of dimension $d$, then ${\rm skel}_k(X)$ is $\FF$-tight for $1\leq
k\leq d$.
\end{lemma}

\begin{proof}
Since $k\geq 1$ and $X$ is neighbourly by Lemma \ref{lemma:2.2}, it follows that ${\rm skel}_k(X)$ is neighbourly
and hence connected. Let $Y$ be an induced subcomplex of ${\rm skel}_k(X)$. Then $Y = {\rm skel}_k(W)$, where $W$
is an induced subcomplex of $X$. Since $X$ is $\FF$-tight, it follows that, for $0\leq i\leq k-1$, $H_i(Y) =
H_i(W) \to H_i(X) = H_i({\rm skel}_k(X))$ is injective. Clearly, $Z_k(Y) \subseteq Z_k({\rm skel}_k(X))$. Since
both $Y$ and ${\rm skel}_k(X)$ are of dimension $\leq k$, $B_k(Y) = 0 = B_k({\rm skel}_k(X))$. Therefore, $H_k(Y)
\to H_k({\rm skel}_k(X))$ is injective.
\end{proof}

\begin{lemma} \label{lemma:2.5}
Every $\FF$-tight, closed, triangulated manifold is $\FF$-orientable.
\end{lemma}

\begin{proof}
Let $X$ be an
$\FF$-tight, closed, triangulated  $d$-manifold. We can assume that $d\geq 2$. Since $X$ is a connected, closed,
triangulated $d$-manifold, it is easy to see ab initio that, for any proper subcomplex $Y$ of $X$, $H_d(Y; \FF) =
0$. Now fix a vertex $x$ of $X$, and consider the induced subcomplex $X^x$ of $X$ on the complement of $x$. Then
$X = X^x \cup (x \ast X_x)$ and $X^x \cap (x \ast X_x) = X_x$. Since $H_d(X^x; \FF) = 0$, $H_{d-1}(x \ast X_x;
\FF)=0$, $H_{d-1}(X_x; \FF) = \FF$ and $H_{\ast}(X^x; \FF) \to H_{\ast}(X; \FF)$ is injective, it follows from
the exactness of the Mayer-Vietoris sequence
\begin{align}
\cdots \to & H_d(X^x; \FF) \oplus H_d(x \ast X_x; \FF) \to H_d(X; \FF) \to H_{d-1}(X_x; \FF)
\nonumber \\
  \to & H_{d-1}(X^x; \FF) \oplus H_{d-1}(x \ast X_x; \FF) \to H_{d-1}(X; \FF) \to \cdots \nonumber
\end{align}
that $H_d(X; \FF) = \FF$.
\end{proof}

Let $X$ be a simplicial complex of dimension $d$. In \cite{BDStellated}, the first two authors have defined the sigma vector
$(\sigma_0, \sigma_1, \dots, \sigma_d)$ of $X$ with respect to a field $\FF$ as
\begin{align}
\sigma_i = \sigma_i(X; \FF) & := \sum_{A\subseteq V(X)} \tilde{\beta}_i(X[A])\left/\binom{f_0(X)}{\#(A)}\right.,
\, 0\leq i \leq d,  \nonumber
\end{align}
and $\sigma_i(X;\FF) = 0$ for $i>\dim(X)$. Here $\#(A)$ denotes the number of vertices in $A$. Also, we have
adopted the convention that $\widetilde{\beta}_0(\{\emptyset\}) = - 1$ and $\widetilde{\beta}_i(\{\emptyset\}) =
0$ if $i > 0$. Note that, in consequence, $\sigma_0(X)$ may be a negative rational number. In this paper, we
introduce the following normalization of the sigma vector which will be referred to as the {\em sigma-star}
vector.
\begin{align}
\sigma_i^{\ast} = \sigma_i^{\ast}(X; \FF) & :=  \frac{\sigma_i(X; \FF)}{1+f_0(X)},  \, \, i \geq 0. \nonumber
\end{align}
(See Theorem \ref{theorem:5.7} for a justification of this normalization.) In \cite{BaMu}, the first author
introduced the mu-vector $(\mu_0, \dots, \mu_d)$ with respect to $\FF$ of a $d$-dimensional simplicial complex
as follows.
\begin{align}
\mu_0 = \mu_0(X; \FF) & :=  \sum_{x\in V(X)} \frac{1}{1+f_0(X_x)},   \nonumber \\
\mu_i = \mu_i(X; \FF) & := \delta_{i1}\mu_0(X; \FF) +  \sum_{x\in V(X)} \sigma^{\ast}_{i-1}(X_x; \FF),  \, \, 1
\leq i \leq d. \nonumber
\end{align}
Here $\delta_{i1}$ denotes Kronecker's symbol. Thus, $\delta_{i1}=1$ if $i=1$, and $=0$ otherwise. Notice that
$\mu_1(X; \FF)$ is independent of the field $\FF$. Therefore, we will write $\mu_1(X)$ for $\mu_1(X; \FF)$. We
have the following

\begin{lemma} \label{lemma:2.6}
Let $M$ be an $\FF$-orientable, neighbourly, closed, triangulated $3$-manifold. Then $\beta_1(M; \FF) \leq
\mu_1(M)$. Equality holds here if and only if $M$ is $\FF$-tight.
\end{lemma}

\begin{proof}
The inequality (as well as the fact that  equality holds if $M$ is $\FF$-tight) is a special case of
\cite[Corollary 1.9]{BaMu}. Suppose $\mu_1 = \beta_1$. Since $M$ is 2-neighbourly, it also satisfies $\mu_0 = 1 =
\beta_0$. By \cite[Theorem 1.7]{BaMu}, $\mu_2 = \mu_1$ and $\mu_3=\mu_0$. Also, since $M$ is $\FF$-orientable,
Poincar\'{e} duality implies $\beta_2=\beta_1$ and $\beta_3 = \beta_0$. Therefore, $\mu_2=\beta_2$ and
$\mu_3=\beta_3$. Hence \cite[Corollary 1.9]{BaMu} implies that $M$ is $\FF$-tight.
\end{proof}

Recall that a closed, triangulated $d$-manifold $X$ is called {\em orientable} if $H_d(X; \ZZ) \neq 0$. It
follows from the universal coefficient theorem that, for a field $\FF$ of odd characteristic, a closed,
triangulated manifold is $\FF$-orientable if and only if it is orientable.

\begin{corollary} \label{coro:2.7}
Let $p$ be a  prime and let $M$ be an orientable, closed, triangulated $3$-manifold. If $M$ is $\ZZ_p$-tight but
not $\QQ$-tight, then $p$ divides the order of the torsion subgroup of $H_1(M; \ZZ)$.
\end{corollary}

\begin{proof}
The hypothesis and Lemma \ref{lemma:2.6} imply that $\beta_1(M; \QQ) < \mu_1(M) = \beta_1 (M; \ZZ_p)$. Hence the
result follows from the universal coefficient theorem.
\end{proof}


The following lemma is immediate from the definition of tightness.

\begin{lemma} \label{lemma:2.8}
{\rm (a)} A simplicial complex is tight with respect to a field of characteristic $p$ if and only if it is
$\ZZ_p$-tight. {\rm (b)} A simplicial complex is tight with respect to a field of characteristic zero if and only
if it is $\ZZ_p$-tight for all primes $p$.
\end{lemma}

For any non-empty finite set $\alpha$, $\overline{\alpha}$ will denote the simplicial complex whose faces are all
the subsets of $\alpha$. Thus, if $\#(\alpha) = d+1$, $\overline{\alpha}$ is the standard triangulation of the
$d$-ball (namely, it is the face complex of the geometric $d$-simplex). If $\#(\alpha) = d+2$, then the boundary
$\partial \overline{\alpha}$ of $\overline{\alpha}$ is the standard triangulation of the $d$-sphere; it will be
denoted by $S^{\,d}_{d + 2}$. (More generally, $S^{\,d}_n$ usually denotes an $n$-vertex triangulation of the
$d$-sphere.) From the definition of a stacked sphere, one can deduce the following (this also follows from
\cite[Lemmas 4.3 (b) \& 4.8 (b)]{BDLBT}).

\begin{lemma} \label{lemma:stacked_sphere}
A simplicial complex $S$ is a stacked $d$-sphere if and only if $S$ is a connected sum of finitely many copies of
the $(d+2)$-vertex standard sphere $S^{\,d}_{d + 2}$.
\end{lemma}

Let $X$ be a closed, triangulated manifold and let $\sigma$ and $\tau$ be facets of $X$. For a bijection $\psi:
\sigma \to \tau$, let $X^\psi$ be the simplicial complex obtained from $X \setminus \{ \sigma, \tau\}$ by
identifying $v$ and $\psi(v)$ for $v \in \sigma$. If ${\rm lk}_X(v) \cap {\rm lk}_X(\psi(v))=\{\emptyset\}$ for
each vertex $v \in \sigma$, then $X^\psi$ is a triangulated manifold. If $\sigma$ and $\tau$ belong to different
connected components, say $\sigma\in X_1$, $\tau \in X_2$ and $X = X_1\sqcup X_2$, then $X^{\psi}$ is said to be
the {\em connected sum} of $X_1$ and $X_2$ and is denoted by $X_1\#_{\psi} X_2$. If $\sigma$ and $\tau$ belong to
the same connected component of $X$, then $X^\psi$ is said to be obtained from $X$ by a {\em combinatorial handle
addition}. For $d\geq 2$, we recursively define the class $\mathcal{H}^{d+1}(k)$ as follows. $\mathcal
H^{d+1}(0)$ is the set of stacked $d$-spheres. A triangulated $d$-manifold is in $\mathcal H^{d+1}(k +1)$ if it
is obtained from a member of $\mathcal H^{d+1}(k)$ by a combinatorial handle addition. The \textit{Walkup's
class} {$\mathcal H^{d+1}$} is the union $\mathcal H^{d+1}=\bigcup_{k \geq 0} \mathcal H^{d+1}(k)$. In \cite{Ka},
Kalai proved

\begin{proposition}[Kalai] \label{prop:Kalai}
Let $M$ be a connected, closed, triangulated manifold of dimension $d\geq 4$. Then $M$ is locally stacked if and
only if $M$ is in $\mathcal H^{d+1}$.
\end{proposition}

Example \ref{exam:Lutz} shows that Proposition \ref{prop:Kalai} does not hold in dimension 3. From \cite{DM} we
know the following.

\begin{proposition}[Datta-Murai] \label{prop:Walkup_class}
Let $M$ be a connected, closed, triangulated manifold of dimension $d\geq 2$. Then $M$ is stacked if and only if
$M$ is in $\mathcal H^{d+1}$.
\end{proposition}

\section{Induced surfaces in tight triangulations}

Let $X_x$ and $X^x$ be as in Notation \ref{notation1}. If $x\neq y$ are two vertices of a simplicial complex $X$,
then we shall also use notations such as $X^x_y$ for $(X^x)_y = (X_y)^x$. We say that two vertices in a
simplicial complex $X$ are adjacent (or, that they are neighbours) if they form an edge of $X$. We now
introduce\,:

\begin{notation} \label{notation:3.1}
{\rm If $x\neq y$ are vertices of a simplicial complex $X$, then $c_X(x, y)$ will denote the number of distinct
connected components $K$ of $X^x_y$ such that $x$ is adjacent in $X_y$ with some vertex in $K$.  }
\end{notation}

\begin{lemma} \label{lemma:3.2}
If $X$ is an $\FF$-tight simplicial complex then for all $x\in V(X)$, we have $\beta_1(X; \FF) =
\tilde{\beta_0}(X_x; \FF) + \beta_1(X^x; \FF)$.
\end{lemma}

\begin{proof}
Clearly, $X = X^x \cup (x \ast X_x)$ and $X^x \cap (x \ast X_x) = X_x$. Therefore, the Mayer-Vietoris theorem
yields the exact sequence (noting that the cone $x\ast X_x$ is homologically trivial)
\begin{align}
H_1(X^x; \FF) & \to H_1(X; \FF) \to \widetilde{H}_{0}(X_x; \FF) \to \widetilde{H}_{0}(X^x; \FF). \nonumber
\end{align}
Since $X^x$ is an induced subcomplex of the $\FF$-tight complex $X$, the map $H_1(X^x; \FF) \to H_1(X; \FF)$ is
injective. Lemma \ref{lemma:2.2} implies that $X^x$ is connected and hence $\widetilde{H}_{0}(X^x; \FF) = 0$. So
we have the short exact sequence
\begin{align}
0 \to & H_1(X^x; \FF) \to H_1(X; \FF) \to \widetilde{H}_{0}(X_x; \FF) \to 0. \nonumber
\end{align}
Hence we get the result.
\end{proof}

Lemmas \ref{lemma:2.2} and \ref{lemma:2.3} show that tightness is a severe structural constraint on a simplicial
complex. So it is surprising that, beyond these two lemmas, no further structural (combinatorial) consequence of
tightness seems to have been known. The following lemma establishes a strong structural restriction on the
2-skeleton of an $\FF$-tight simplicial complex.

\begin{lemma} \label{lemma:3.3}
Let $X$ be an $\FF$-tight simplicial complex for some field $\FF$. Then, for any two distinct vertices $x, y$ of
$X$, we have $c_X(x, y) = c_X(y, x)$.
\end{lemma}

\begin{proof}
By Lemma \ref{lemma:3.2}, $\beta_1(X) = \tilde{\beta_0}(X_x) + \beta_1(X^x)$. Since $X^x$ is also tight by Lemma
\ref{lemma:2.3}, applying Lemma \ref{lemma:3.2} to the vertex $y$ of $X^x$, we get $\beta_1(X^x) =
\tilde{\beta_0}(X_y^x) + \beta_1(X^{xy})$. Therefore, $\beta_1(X) = \tilde{\beta_0}(X_x) + \tilde{\beta_0}(X_y^x)
+ \beta_1(X^{xy})$. Interchanging the vertices $x$ and $y$ in this argument yields $\beta_1(X) =
\tilde{\beta_0}(X_y) + \tilde{\beta_0}(X_x^y) + \beta_1(X^{yx})$. Since $X^{xy} = X^{yx}$, we get
\begin{align}
\tilde{\beta_0}(X^x_y) & - \tilde{\beta_0}(X_y) = \tilde{\beta_0}(X_x^y) - \tilde{\beta_0}(X_x). \nonumber
\end{align}
But the two sides of this equation are just one less than $c_X(x, y)$ and $c_X(y, x)$. Hence the result.
\end{proof}

In the following, we will make use of some basic concepts from graph theory, where in our context a {\em graph}
can be seen as a simplicial complex of dimension $\leq 1$. In this paper, we do not use any non-trivial results
from graph theory, but the language and the geometric intuition of graph theory will be useful. Recall that the
{\em degree} of a vertex $v$ (denoted by $\deg(v)$) in a simplicial complex is the number of edges (1-faces)
through $v$. A graph is said to be {\em regular}, if all its vertices have the same degree. For $n\geq 3$, the
{\em cycle of length $n$} (in short, {\em $n$-cycle}) is the unique connected regular graph of degree two on $n$
vertices. It is the unique $n$-vertex triangulation of the circle $S^1$. An $n$-cycle with edges $a_1a_2, \dots,
a_{n-1}a_n, a_na_1$ will be denoted by $C_n(a_1, \dots, a_n)$. For $n\geq 1$, the {\em path of length $n$} (the
$n$-path) is the antistar of a vertex in the $(n+2)$-cycle. By an {\em induced cycle} (resp., path) in a
simplicial complex $X$, we mean an induced subcomplex of $X$ which is a cycle (resp., path). Notice that, in
particular, a 3-cycle is induced in $X$ if and only if it does not bound a triangle (2-face) in $X$. When $n \geq
4$, an $n$-cycle is induced in X if and only if it is induced in the graph ${\rm skel}_1(X)$. A connected acyclic
graph is called a {\em tree}.

\begin{lemma} \label{lemma:3.4}
Let the link of some vertex $x$ in a $2$-dimensional $\FF$-tight simplicial complex $X$ be a cycle. Then $X$ is a
triangulation of a closed $2$-manifold.
\end{lemma}

\begin{proof}
Let $C = X_x$ be a cycle. Fix a vertex $y\neq x$ of $X$. It suffices to show that the link $X_y$ is also a cycle.
Note that, since $X$ is neighbourly (Lemma \ref{lemma:2.2}), $y$ is a vertex of $C$. Let $z$ and $w$ be the two
neighbours of $y$ in $C$. It follows that $z$ and $w$ are the only two neighbours of the vertex $x$ in the graph
$X_y$. Therefore, it suffices to show that $X^x_y$ is a path joining $z$ and $w$.

Since $X_x = C$ is a cycle and $X^y_x = C^y$ is a path, they are both connected. So $c_X(y, x) = 1$. Therefore,
by Lemma \ref{lemma:3.3}, $c_X(x, y) = 1$. That is, the vertices $z$ and $w$ (being the two neighbours of $x$ in
$X_y$) belong to the same component of $X^x_y$. Thus, there is a path in $X^x_y$ joining $z$ to $w$. Let $P$ be a
shortest path in the graph $X^x_y$ joining $z$ to $w$. Then, $P$ is an induced path in $X^x_y$. Take any vertex
$v \neq x, y, z, w$ in $X$. (If there is no such vertex then $X^x_y$ is the edge $zw$, and we are done.) We will
show that $v\in P$.
Look at the induced subcomplex $Y = X^v$ of $X$. Then $Y_x = X^v_x = C^v$ is a path in which $y$ is an interior
vertex. So $Y^y_x = C^{vy}$ is the disjoint union of two paths. The vertex $z$ belongs to one of these two paths
and $w$ belongs to the other. Therefore, $c_Y(y, x) = 2$. By Lemma \ref{lemma:2.3}, $Y$ is also $\FF$-tight. So,
by Lemma \ref{lemma:3.3}, $c_Y(x, y) =2$. That is, the neighbours $z$ and $w$ of $x$ in $Y_y$ belong to different
components of $Y^x_y$. Therefore, $v$ belongs to the path $P$ (or else $P$ would be a path in $Y^x_y$ joining $z$
and $w$). Since $v \neq z, w$ was an arbitrary vertex of $X^x_y$, this shows that the path $P$ is a spanning path
in $X^x_y$ (i.e., it passes through all the vertices). Since $P$ is also an induced path in $X^x_y$, it follows
that $X^x_y = P$ is a path joining $z$ and $w$.
\end{proof}

The following result is a straightforward consequence of Lemma \ref{lemma:3.4}. Notice that there is no
restriction on the dimension of $M$ in this result, and $M$ need not be a triangulated manifold.

\begin{theorem} \label{theorem:3.5}
Let the simplicial complex $M$ be tight with respect to some field. Let $C$ be an induced cycle in the link of a
vertex $x$ in $M$. Then the induced subcomplex of $M$ on the vertex set of the cone $x\ast C$ is a closed,
triangulated $2$-manifold.
\end{theorem}

\begin{proof}
Let $Y = M[V(C) \cup\{x\}]$ and let $X = {\rm skel}_2(Y)$. Then, by Lemmas \ref{lemma:2.3} and \ref{lemma:2.4},
$X$ is $\FF$-tight. Clearly, $X$ is two dimensional and $X_x = C$ is a cycle. Thus, by Lemma \ref{lemma:3.4}, $X$
is a closed, triangulated 2-manifold. To complete the proof, it suffices to show that $Y = X$, i.e., that
$\dim(Y) =2$. Suppose, if possible, $\dim(Y) > 2$. Take a 3-face $\alpha\in Y$. Then the induced subcomplex of
$X= {\rm skel}_2(Y)$ on the vertex set $\alpha$ is a 4-vertex triangulated 2-sphere. Since $X$ is a connected,
closed, triangulated 2-manifold, it follows that $X = S^2_4$, and $V(X) = \alpha$. Thus $C$ is the 3-cycle on the
vertex set $\alpha\setminus\{x\}$. But, the 2-face $\alpha\setminus\{x\}$ is in $Y_x\subseteq M_x$. This
contradicts the assumption that $C$ is an induced cycle in the link $M_x$.
\end{proof}

\begin{corollary} \label{coro:3.6}
Let $S$ be the link of some vertex in an $\FF$-tight simplicial complex $M$.
\begin{enumerate}[{\rm (a)}]
\item If ${\rm char}(\FF) = 2$, then $S$ has no induced cycle of length $\equiv 1$ $($mod $3)$.

\item If ${\rm char}(\FF) \neq 2$, then $S$ has no induced cycle of length $\equiv 0, 1, 4, 5, 7, 8, 9$ or $10$
$($mod $12)$.
\end{enumerate}
\end{corollary}

\begin{proof}
Let $x\in V(M)$ and let $C$ be an induced $n$-cycle in $S = M_x$. By Theorem \ref{theorem:3.5}, the induced
subcomplex $X = M[V(C)\cup\{x\}]$ is an $(n+1)$-vertex, closed, triangulated 2-manifold. By Lemma
\ref{lemma:2.2}, $X$ is neighbourly. So it has $n+1$ vertices, $n(n+1)/2$ edges and hence $n(n+1)/3$ triangles.
So 3 divides $n(n+1)$, i.e., $n \not \equiv 1$ (mod 3). This proves part (a).

If ${\rm char}(\FF) \neq 2$ then, by Lemmas \ref{lemma:2.3} and \ref{lemma:2.5}, $X$ is an orientable,
triangulated 2-manifold. So its Euler characteristic $\chi(X) = (n+1)(6-n)/6$ is an even number. Thus, $n
\not\equiv 0, 1, 4, 5, 7, 8, 9$ or 10 (mod 12). This proves part (b).
\end{proof}

\section{Odd characteristic}

In this section, we prove that any triangulated $3$-manifold is tight with respect to a field of odd
characteristic if and only if it is neighbourly, orientable and stacked. For this, we first need some additional
preliminary results on triangulations of the 2-sphere.

\begin{lemma} \label{lemma:4.1}
Let $S$ be a triangulation of $S^{\hspace{.15mm}2}$. If $S$ has no induced cycle of length $\leq 5$ then $S =
S^{\hspace{.15mm}2}_4$.
\end{lemma}

\begin{proof}
Let $x$ be a vertex of minimum degree in $S$. It is well known (and easy to prove) that the minimum degree of any
triangulation of $S^{\hspace{.15mm}2}$ is at most five. So $\deg(x) = 3, 4$ or 5.

If $\deg(x) = 4$ or 5, then $S_x$ is a 4-cycle or a 5-cycle in $S$, and hence it is not induced. So there are
vertices $y, z$ in $S_x$ such that $yz$ is an edge in $S$ but not in $S_x$. So $C_3(x, y, z)$ is an induced
3-cycle in $S$, a contradiction. Thus, $\deg(x) = 3$. Say $S_x$ is the 3-cycle $C_3(x_1, x_2, x_3)$. Since this
3-cycle is not an induced cycle in $S$, it follows that $x_1x_2x_3\in S$. Then $(x\ast S_x) \cup
\{x_1x_2x_3\}=S^{\hspace{.15mm}2}_4$. Since a triangulation of $S^{\hspace{.15mm}2}$ can not be a proper
subcomplex of another triangulation of $S^{\hspace{.15mm}2}$, it follows that $S=S^{\hspace{.15mm}2}_4$.
\end{proof}

\begin{definition} \label{def:primitive}
{\rm A triangulated $d$-sphere $S$ is said to be {\em primitive} if it can not be written as a connected sum of
two triangulated $d$-spheres.}
\end{definition}

Clearly, every triangulated $d$-sphere is a connected sum of finitely many primitive $d$-spheres.

\begin{lemma} \label{lemma:4.3}
Let $S$ be a triangulated $d$-sphere. Then $S$ is primitive if and only if $S$ has no induced subcomplex
isomorphic to $S^{\hspace{.15mm}d-1}_{d+1}$.
\end{lemma}

\begin{proof}
If $S$ is not primitive, then $S = S_1 \# S_2$, where $S_1$ and $S_2$ are triangulated $d$-spheres. Let $\alpha$
be the unique common facet of $S_1$ and $S_2$. Then the boundary of $\overline{\alpha}$ is an induced
$S^{\hspace{.15mm}d-1}_{d+1}$ in $S$.

Conversely, suppose $S$ has an induced $S^{\hspace{.15mm}d-1}_{d+1}$, say with vertex set $\alpha$. This
$S^{\hspace{.15mm}d-1}_{d+1}$ divides $S$ into two triangulated $d$-balls $B_1$ and $B_2$ such that $S = B_1\cup
B_2$, $\partial B_1 = \partial B_2 = B_1\cap B_2 = S^{\hspace{.15mm}d-1}_{d+1}$. Put $S_i = B_i\cup\{\alpha\}$,
$i=1,2$. Then $S_1$, $S_2$ are triangulated $d$-spheres and $S = S_1 \# S_2$. So $S$ is not primitive.
\end{proof}

The following lemma and definition clarify the meaning of connected sums of several primitive triangulated
spheres.

\begin{notation} \label{notation:c_sum}
{\rm Let $S$ be a triangulation of $S^{\hspace{.15mm}d}$. Let $\mathcal{A}(S) := \{\alpha\subseteq V(S) \, : \,
S[\alpha] \cong S^{\hspace{.15mm}d-1}_{d+1}\}$, $\overline{S} := S \cup \mathcal{A}(S)$, and let $\mathcal{B}(S)$
be the collection of all the induced subcomplexes of the simplicial complex $\overline{S}$ which are primitive
triangulations of $S^{\hspace{.15mm}d}$. Let $\mathcal{T}(S)$ be the graph with vertex set $\mathcal{B}(S)$ such
that $S_1, S_2\in \mathcal{B}(S)$ are adjacent in $\mathcal{T}(S)$ if  $S_1 \cap S_2 = \overline{\alpha}$ for
some $\alpha \in \mathcal{A}(S)$. }
\end{notation}

\begin{lemma} \label{lemma:c_sum}
For any triangulated $d$-sphere $S$ the following hold.
\begin{enumerate}[{\rm (a)}]
\item Any two members of $\mathcal{B}(S)$ have at  most one common facet; if they have a common facet $\alpha$
then $\alpha \in \mathcal{A}(S)$. \item Each member of $\mathcal{A}(S)$ belongs to exactly two members of
$\mathcal{B}(S)$. \item There is a natural bijection from $\mathcal{A}(S)$ onto the set of edges of
$\mathcal{T}(S)$. It is given by $\alpha \mapsto \{S_1, S_2\}$, where, for $\alpha\in \mathcal{A}(S)$, $S_1$ and
$S_2$ are the two members of $\mathcal{B}(S)$ containing $\alpha$. \item The graph $\mathcal{T}(S)$ is a tree. In
consequence, $\#\mathcal{B}(S) = 1 + \#\mathcal{A}(S)$.
\end{enumerate}
\end{lemma}

\begin{proof}
Induction on $m : = 1 + \#\mathcal{A}(S)$. If $m=1$, $\mathcal{A}(S)$ is empty, so that $S$ is primitive by Lemma
\ref{lemma:4.3}. So let $m >1$, and suppose that the result holds for all smaller values of $m$. In this case,
$\mathcal{A}(S)$ is non-empty. Take $\alpha\in\mathcal{A}(S)$. By the proof of Lemma \ref{lemma:4.3}, there are
triangulated $d$-spheres $S^{\prime}$, $S^{\prime\prime}$ such that $S^{\prime} \cap S^{\prime\prime} =
\overline{\alpha}$, $S = (S^{\prime} \cup S^{\prime\prime}) \setminus\{\alpha\}$. Thus, $S^{\prime},
S^{\prime\prime} \subseteq S \cup\{\alpha\} \subseteq \overline{S}$. Hence it is easy to see that $\mathcal{A}(S)
= \mathcal{A}(S^{\prime}) \sqcup \mathcal{A}(S^{\prime\prime}) \sqcup \{\alpha\}$, $\mathcal{B}(S) =
\mathcal{B}(S^{\prime}) \sqcup \mathcal{B}(S^{\prime\prime})$ and the graph $\mathcal{T}(S)$ is obtained from the
disjoint union of the graphs $\mathcal{T}(S^{\prime})$  and $\mathcal{T}(S^{\prime\prime})$ by adjoining a single
edge from a vertex of $\mathcal{T}(S^{\prime})$ to a vertex of $\mathcal{T}(S^{\prime\prime})$. By induction
hypothesis, the result is valid for $S^{\prime}$ and $S^{\prime\prime}$, and its validity for $S$ follows. This
completes the induction.
\end{proof}

\begin{definition} \label{def:c_sum}
{\rm Let $S$ be a triangulated $d$-sphere.  Any tree has a leaf (a vertex of degree one) and the deletion of a
leaf from a non-trivial tree leaves a subtree. Therefore, the members of $\mathcal{B}(S)$ may be ordered as $S_1,
\dots, S_m$ in such a way that, for each $i$, $1\leq i\leq m$, $S_i$ is a leaf of the induced subtree of
$\mathcal{T}(S)$ on $\{S_1, \dots, S_i\}$ (when $m\geq 2$, this ordering is not unique). For any such ordering,
we write $S = S_1\# S_2 \# \cdots \# S_m$, and say that $S$ is the {\em connected sum} of the $S_i$'s. Clearly,
we have $S_1 \# \cdots \# S_m = (S_1 \# \cdots \# S_{m-1}) \# S_m$. }
\end{definition}

As a special case of \cite[Theorem 8.5]{Ka}, Kalai proved a nice characterization of stacked 2-spheres. A
triangulated 2-sphere $S$ is stacked if and only if $S$ has no induced cycle of length $\geq 4$. The following
result is an improvement on this characterization.

\begin{theorem} \label{theorem:4.7}
Let $S$ be a triangulated $2$-sphere. Then $S$ is stacked if and only if it has no induced cycle of length $4$ or
$5$.
\end{theorem}

\begin{proof}
Write $S = S_1\# S_2\# \cdots\# S_k$, where the $S_i$'s are primitive 2-spheres.

If $S$ is stacked then, by Lemma \ref{lemma:stacked_sphere}, each $S_i$ is a copy of $S^{\hspace{.15mm}2}_4$. Let
$C$ be an induced cycle of length $\geq 4$ in $S$. Since an induced cycle of length $\geq 4$ in a connected sum
of two triangulated manifolds must be an induced cycle in one of the summands, it follows inductively that $C$ is
an induced cycle in one of the $S_i$'s. Since $S^{\hspace{.15mm}2}_4$ has no induced cycle at all, it follows
that $S$ has no induced cycle of length $\geq 4$. This proves the ``only if" part.

For the converse, assume that $S$ has no induced cycle of length 4 or 5. It follows that no $S_i$ has any induced
cycle of length 4 or 5. Being primitive, $S_i$ has no induced cycle of length 3 either (Lemma \ref{lemma:4.3}).
Thus, by Lemma \ref{lemma:4.1}, each $S_i$ is a copy of $S^{\hspace{.15mm}2}_4$. Therefore, by Lemma
\ref{lemma:stacked_sphere}, $S$ is stacked. This proves the ``if" part.
\end{proof}

Now we are ready to prove one of the main results of this paper.

\begin{theorem} \label{theorem:4.8}
A closed, triangulated $3$-manifold $M$ is tight with respect to some field $\FF$ with ${\rm char}(\FF) \neq 2$
if and only if $M$ is orientable, neighbourly and stacked.
\end{theorem}

\begin{proof}
The ``if" part follows from Proposition \ref{prop:BDTh2.24}. To prove the ``only if" part, let $M$ be an
$\FF$-tight, closed, triangulated 3-manifold, ${\rm char}(\FF) \neq 2$. By Lemmas \ref{lemma:2.2} and
\ref{lemma:2.5}, $M$ is neighbourly and orientable. It remains to show that $M$ must be stacked.

Corollary \ref{coro:3.6} (b) shows that, for each vertex $x$ of $M$, the vertex link $M_x$ is a triangulated
2-sphere with no induced 4-cycle or 5-cycle. So, Theorem \ref{theorem:4.7} implies that each $M_x$ is a stacked
2-sphere. Thus, $M$ is locally stacked. The result now follows by Proposition \ref{prop:BDTh2.24}.
\end{proof}

\begin{corollary} \label{coro:4.9}
Let $M$ be a closed, triangulated $3$-manifold. If $M$ is tight with respect to a field $\FF$, where ${\rm
char}(\FF) \neq 2$, then $M$ is strongly minimal.
\end{corollary}

\begin{proof}
This result follows from Theorem \ref{theorem:4.8} and Proposition \ref{prop:sminimal}.
\end{proof}

\begin{corollary} \label{coro:4.10}
Let $M$ be a closed, triangulated $3$-manifold. If $M$ is tight with respect to a field of odd characteristic
then $M$ is tight with respect to all fields.
\end{corollary}

\begin{proof}
Let $M$ be $\FF$-tight, where ${\rm char}(\FF)$ is odd. Then, by Theorem \ref{theorem:4.8}, $M$ is orientable,
neighbourly and stacked. The result now follows from Proposition \ref{prop:BDTh2.24}.
\end{proof}

\begin{corollary} \label{coro:new}
Let $M$ be a closed, triangulated $3$-manifold. If $M$ is tight with respect to a field of odd characteristic
then either {\rm (i)} $M = S^{\hspace{0.15mm}3}_5$ or {\rm (ii)} $\beta :=(f_0(M)-4)(f_0(M)-5)/20$ is an integer
greater than one and $M$ triangulates $(S^{\hspace{0.15mm}2} \times S^{\hspace{0.1mm}1})^{\#\beta}$.
\end{corollary}

\begin{proof}
Let $M$ be $\FF$-tight, where ${\rm char}(\FF)$ is odd. Assume that $M \neq S^{\hspace{0.15mm}3}_5$. Since the
only tight triangulation of the sphere of dimension $d\geq 1$ is the standard sphere $S^{\hspace{0.15mm}d}_{d
+2}$, it follows that $M$ is not a triangulated $3$-sphere.

Since ${\rm char}(\FF) \neq 2$, by Theorem \ref{theorem:4.8}, $M$ is orientable, neighbourly and stacked.
Therefore, by Proposition \ref{prop:Walkup_class}, $M$ triangulates $(S^{\hspace{0.15mm}2} \times
S^{\hspace{0.1mm}1})^{\#k}$ for some positive integer $k$. Then, $\beta_1(M; \FF) = k$.

Since $M$ is stacked, by Proposition \ref{prop:BDTh2.24}, $(f_0(M)-4)(f_0(M)-5)= 20\beta_1(M; \FF) = 20k$. So
$(f_0(M)-4)(f_0(M)-5)/20 =k$. Thus, $\beta =(f_0(M)-4)(f_0(M)-5)/20$ is a positive integer and $M$ triangulates
$(S^{\hspace{0.15mm}2} \times S^{\hspace{0.1mm}1})^{\#\beta}$. Since there is no 9-vertex triangulation of
$S^{\hspace{0.15mm}2} \times S^{\hspace{0.1mm}1}$ (see \cite{Wa}), it follows that $\beta \neq 1$. This completes
the proof.
\end{proof}

\section{Characteristic two} 

By Lemma \ref{lemma:2.8}, a simplicial complex is tight with respect to a field of characteristic two if and only
if it is $\ZZ_2$-tight. So, without loss of generality, we take $\FF= \ZZ_2$ in this section. However, all the
results apply equally well to arbitrary fields of characteristic two. Here, we characterize the links of
$\ZZ_2$-tight, triangulated $3$-manifolds. This characterization is important since it leads to (a) severe
restrictions on the size and topology of such a triangulation (see Theorem \ref{theorem:5.1}, Corollary
\ref{coro:5.12} and the tables at the end of this section) and (b) a polynomial time algorithm to decide
tightness of $3$-manifolds which is described in detail in \cite{BBDSS}.

By Proposition \ref{prop:BDTh2.24}, all stacked, neighbourly, closed, triangulated 3-manifolds are $\ZZ_2$-tight.
Here is a partial converse. (In case $M$ is orientable, this result follows from Corollary \ref{coro:2.7}.)

\begin{theorem} \label{theorem:5.1}
Let $M$ be a $\ZZ_2$-tight, closed, triangulated $3$-manifold. If the torsion subgroup of $H_1(M; \ZZ)$ is of odd
order $($possibly trivial$)$, then $M$ is stacked $($and neighbourly$)$.
\end{theorem}

\begin{proof}
Assume, on the contrary, that $M$ is not stacked. Then, by Proposition \ref{prop:BDTh2.24}, $M$ is not locally
stacked. So there exists a vertex $v$ whose link $M_v$ is not a stacked 2-sphere. By Theorem \ref{theorem:4.7}
and Corollary \ref{coro:3.6} (a), $M_v$ has an induced cycle $C$ of length 5. Then, by Theorem \ref{theorem:3.5},
the induced subcomplex $N:= M[\{v\}\cup V(C)]$ of $M$ is a closed, triangulated 2-manifold. Since $N$ is an
induced subcomplex of $M$, by Lemmas \ref{lemma:2.2} and \ref{lemma:2.3}, $N$ is neighbourly. Thus, $N$ is a
6-vertex, neighbourly, triangulated 2-manifold and hence is the 6-vertex triangulation $\mathbb{RP}^2_6$ of
$\mathbb{RP}^2$. Take a non-triangle $abc$ of $N = \mathbb{RP}^2_6$ (there are 10 of them). Then $\alpha= ab+bc+
ca$ can be viewed as a 1-cycle of $N$ with $\ZZ_2$-coefficient as well as with $\ZZ$-coefficient. In both views,
it is not a boundary. However $2\alpha$ is a boundary with $\ZZ$-coefficient. Since the map $H_1(N; \ZZ_2) \to
H_1(M; \ZZ_2)$, induced by the inclusion map $N \hookrightarrow M$, is injective, it follows that $[\alpha] \neq
0$ as an element of $H_1(M; \ZZ_2)$ and hence also as an element of $H_1(M; \ZZ)$. But, $2[\alpha] = 0$ in
$H_1(M; \ZZ)$. So $[\alpha]$ is an element of order 2 in $H_1(M; \ZZ)$. This is a contradiction to the assumption
on $H_1(M; \ZZ)$.
\end{proof}

\begin{lemma} \label{lemma:5.2}
Let $S$ be a primitive triangulation of $S^{\hspace{.15mm}2}$. Suppose $S\neq S^{\hspace{.15mm}2}_4$ and $S$ has
no induced $4$-cycle. Then,
\begin{enumerate}[{\rm (a)}]
\item All vertex links of $S$ are induced cycles in $S$.
\item Any two adjacent vertices of $S$ have exactly two common neighbours in $S$.
\end{enumerate}
\end{lemma}

\begin{proof}
Let $x\in V(S)$. Since $S\neq S^{\hspace{.15mm}2}_4$ is primitive, Lemma \ref{lemma:4.3} implies that $S_x$ can
not be a 3-cycle. Let the vertices $y, z$ of the cycle $S_x$ be neighbours in $S$. Since $S$ is primitive, the
3-cycle $C_3(x, y, z)$ can not be induced in $S$. So the triangle $xyz$ is in $S$. Therefore, $y$ and $z$ are
neighbours in $S_x$. This implies that $S_x$ is induced in $S$. This proves part (a).

Let $x, y$ be adjacent vertices of $S$, and let $z$ be a common neighbour of $x$ and $y$. Then, as the 3-cycle
$C_3(x, y, z)$ is not induced in $S$, the triangle $xyz$ is in $S$. So $z$ is the third vertex of one of the two
triangles of $S$ through $xy$. Thus, $x$ and $y$ have exactly two common neighbours. This proves part (b).
\end{proof}

\begin{lemma} \label{lemma:5.3}
Let $S$ be a primitive triangulation of $S^{\hspace{.15mm}2}$ such that $S$ has no induced cycle of length
$\equiv 1$ $($mod $3)$. Then, either $S= S^{\hspace{.15mm}2}_4$ or all the vertices of $S$ have degree $2$ $($mod
$3)$.
\end{lemma}

\begin{proof}
By assumption and Lemma \ref{lemma:5.2} (a), $V(S) = V_0 \sqcup V_2$, where $V_i$ consists of the vertices of
degree $i$ (mod 3), $i=0,2$. If possible, let $x\in V_0$ and $y\in V_2$ be such that $xy$ is an edge of $S$. By
Lemma \ref{lemma:5.2}, the cycles $S_x$ and $S_y$ are induced in $S$ (of length 0 (mod 3) and 2 (mod 3),
respectively) with exactly two common vertices, say $u$ and $v$. Then $u$ and $v$ are the two neighbours of $x$
in $S_y$ and of $y$ in $S_x$.

\smallskip

\noindent {\em Claim.} No vertex in $V(S_x)\setminus V(S_y)$ is adjacent to any vertex in $V(S_y)\setminus
V(S_x)$.

Indeed, if $a \in S_x$ is a neighbour of $b \in S_y$, then $S$ has the 4-cycle $C_4(x, a, b, y)$. Since $S$ has
no induced 3-cycle or 4-cycle, it follows that one of the triangles $xya$ and $xab$ is in $S$. Hence either $a$
or $b$ is in $S_x \cap S_y$. This proves the claim.

Therefore, if $C$ is the cycle obtained from $S_x\cup S_y$ by deleting the two vertices $x$, $y$ and the four
edges $xu$, $xv$, $yu$, $yv$, then $C$ is an induced cycle in $S$ of length $\equiv 0+2-4 \equiv 1$ (mod 3). This
is a contradiction. Therefore, no vertex in $V_0$ is adjacent to any vertex in $V_2$. Since $S$ is connected and
$V(S) = V_0\sqcup V_2$, it follows that $V_0 =\emptyset$ or $V_2 = \emptyset$. If $V_2 = \emptyset$ then the
degree of each vertex is 0 (mod 3) and hence $S$ has a vertex $z$ of degree 3. Since $S$ is primitive, $S_z$
bounds a triangle. Then, $S$ contains an $S^{\hspace{.15mm}2}_4$ and hence $S= S^{\hspace{.15mm}2}_4$. Otherwise,
$V(S) = V_2$.
\end{proof}

\begin{theorem} \label{theorem:5.4}
Up to isomorphism, $S^{\hspace{.15mm}2}_4$ and $I^2_{12}$ are the only two primitive triangulations of
$S^{\hspace{.15mm}2}$ with no induced cycle of length $\equiv 1$ $($mod $3)$.
\end{theorem}

\begin{proof}
Clearly, $S^{\hspace{.15mm}2}_4$ has no induced cycle whatsoever. It is easy to see that all the induced cycles
of $I^2_{12}$ are 5-cycles. (Indeed, these are precisely the twelve vertex links.) Thus, these two triangulations
of $S^{\hspace{.15mm}2}$ are primitive with no induced cycle of length 1 (mod 3).

Conversely, let $S$ be a primitive triangulation of $S^{\hspace{.15mm}2}$ with no induced cycle of length 1 (mod
3). Assume $S\neq S^{\hspace{.15mm}2}_4$. By Lemma \ref{lemma:5.3}, the minimum degree of the vertices of $S$ is
five. If all its vertices have degree 5, then $S= I^2_{12}$. So, suppose there is a vertex $u$ with $\deg(u) >
5$. By Lemma \ref{lemma:5.3}, $\deg(u) \geq 8$. So $S_u$ is a cycle of length $\geq 8$. Therefore, we may choose
vertices $v_1, v_2, w_1, w_2$ in $S_u$ such that $v_1, v_2$ are at distance 2, $w_1, w_2$ are at a distance 2,
and ${\rm dist}(v_i, w_j) \geq 2$, $i, j = 1, 2$, where all the distances are graphical distances measured along
the cycle $S_u$. Let $D := (u\ast S_u) \cup (v_1\ast S_{v_1}) \cup (v_2\ast S_{v_2}) \cup (w_1\ast S_{w_1})\cup
(w_2\ast S_{w_2})$ (see Figure 1).  The boundary of $D$ is the union of six paths in $S$ (drawn as circular arcs;
edges of $S$ are drawn as straight line segments). Here, the dotted paths might possibly be trivial (of length
zero), but the following argument goes through even in these degenerate cases. Since $S$ has no induced cycle of
length $\leq 4$, it is easy to see that these six paths intersect pairwise at most at common end points, and the
six end points marked in Fig. 1 are distinct vertices of $S$. Thus, the boundary of $D$ is a cycle and hence $D$
is a disc.

\begin{center}
{\includegraphics[height=5cm]{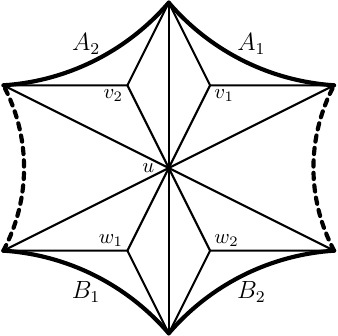}}

Figure 1: Disc $D$ in the proof of Theorem \ref{theorem:5.4}
\end{center}

Fix an index $i \in \{1, 2\}$. Consider the boundary $C_i$ of the disc $(u\ast S_u) \cup (v_i\ast S_{v_i}) \cup
(w_i\ast S_{w_i})$. Then the vertex set of $C_i$ is the union of the vertex sets of $S_u$, $S_{v_i}$, $S_{w_i}$
minus the three vertices $u, v_i, w_i$. By Lemma \ref{lemma:5.2} (b), $S_u$ and $S_{v_i}$ (as also $S_u$ and
$S_{w_i}$) have exactly two common vertices. Since $D$ is a disc, it follows that $S_{v_i}$ and $S_{w_i}$ have a
unique vertex (namely, u) in common. Also, by Lemma \ref{lemma:5.3}, each of $S_u$, $S_{v_i}$, $S_{w_i}$ have
length 2 (mod 3). Therefore, the inclusion exclusion principle shows that the length of $C_i$ is $\equiv 2+2+2 -
(2+2+1) -3 \equiv 1$ (mod 3). Therefore, $C_i$ is not an induced cycle of $S$. Thus, there is an edge $a_ib_i$ in
$S$ such that $a_i$ and $b_i$ are non-consecutive vertices in the cycle $C_i$. By the proof of Lemma
\ref{lemma:5.3}, no vertex of $S_{v_i}\setminus S_u$ is adjacent in $S$ with any vertex of $S_{u}\setminus
S_{v_i}$; also no vertex of $S_{w_i}\setminus S_u$ is adjacent in $S$ with any vertex of $S_{u}\setminus
S_{w_i}$. Therefore, $a_i \in A^{\hspace{-0.25mm}^\circ}_i$, $b_i\in B^{\hspace{-0.15mm}^\circ}_i$ (see Fig. 1).
That is, $a_i$ is an interior vertex of the path $A_i$ and $b_i$ is an interior vertex of the path $B_i$.

Now consider the disc $D^{\prime}$ in $S$ complementary to the disc $D$ (i.e., $D^{\prime} := S[V(S) \setminus
\{u, v_1, v_2, w_1, w_2\}]$). Then, $a_1b_1$ and $a_2b_2$ are in $D^{\prime}$. Clearly, $a_1b_1$ separates
$A^{\hspace{-0.25mm}^\circ}_2$ from $B^{\hspace{-0.25mm}^\circ}_2$ in $D^{\prime}$ (i.e.,
$A^{\hspace{-0.25mm}^\circ}_2$ and $B^{\hspace{-0.25mm}^\circ}_2$ are in different connected components of
$|D^{\prime}|\setminus|a_1b_1|$). Therefore, the geometric edges $|a_1b_1|$ and $|a_2b_2|$ intersect at an
interior point. This contradicts the very definition of the geometric realization of a simplicial complex. Thus,
there is no vertex $u$ of $S$ with $\deg(u) > 5$. This completes the proof.
\end{proof}

\begin{corollary} \label{coro:5.5}
For triangulations $S$ of $S^{\hspace{.15mm}2}$, the following conditions are equivalent\,:
\begin{enumerate}[{\rm (a)}]
\item
$S$ has no induced cycle of length $1$ $($mod $3)$,
\item
all the induced cycles of $S$ have length $3$ and $5$, and
\item
$S$ is a connected sum of $S^{\hspace{.15mm}2}_4$'s and $I^2_{12}$'s.
\end{enumerate}
\end{corollary}
(The statement in (c) includes the possibility that $S$ is either icosian or stacked.)

\begin{proof}
Let $S$ be as in (c). It follows that the only possible induced cycles in the summands  are 5-cycles. Therefore,
the only induced cycles in $S$ are 3-cycles and 5-cycles. Thus, (c) $\Rightarrow$ (b). Trivially, (b)
$\Rightarrow$ (a).

Now suppose (a) holds. Write $S=S_1\# S_2\#\cdots\# S_k$, where each $S_i$ is primitive. Since $S$ has no induced
cycle of length 1 (mod 3), it follows that no $S_i$ has an induced cycle of length 1 (mod 3). Therefore, by
Theorem \ref{theorem:5.4}, each $S_i$ is $S^{\hspace{.15mm}2}_4$ or $I^2_{12}$. Hence $S$ is as in (c). Thus, (a)
$\Rightarrow$ (c).
\end{proof}

Now we can prove the second main result of this paper.

\begin{theorem} \label{theorem:5.6}
Let $M$ be a $\ZZ_2$-tight, closed, triangulated $3$-manifold. Then each vertex link of $M$ is a connected sum of
$S^{\hspace{.15mm}2}_4$'s and $I^2_{12}$'s.
\end{theorem}

\begin{proof}
This is an immediate consequence of Corollary \ref{coro:3.6} (a) and Corollary \ref{coro:5.5}.
\end{proof}

The following result provides a recursive procedure for the computation of the sigma-star vector.

\begin{theorem} \label{theorem:5.7}
Let $X_1$ and $X_2$ be induced subcomplexes of a simplicial complex $X$ and $\FF$ be a field. Suppose $X =
X_1\cup X_2$ and $Y = X_1\cap X_2$. If $Y$ is $k$-neighbourly, $k\geq 2$, then $\sigma_i^{\ast}(X;\FF) =
\sigma_i^{\ast}(X_1; \FF) + \sigma_i^{\ast}(X_2;\FF) - \sigma_i^{\ast}(Y;\FF)$ for $0\leq i\leq k-2$.
\end{theorem}

\begin{proof}
Let us write $m_1 = f_0(X_1)$, $m_2 = f_0(X_2)$, $m = f_0(Y)$. Thus, $f_0(X) = m_1+m_2-m$. For notational
convenience, we write $\tilde{\beta}_i(A)$ for $\tilde{\beta}_i(X[A]; \FF)$,  $A\subseteq V(X)$. Note that any
subset $A$ of $V(X)$ can be uniquely written as $A = A_1\sqcup B \sqcup A_2$, where $A_1 \subseteq V(X_1)
\setminus V(Y)$, $B \subseteq V(Y)$ and $A_2 \subseteq V(X_2) \setminus V(Y)$. Since $Y$ is $k$-neighbourly, it
follows from the exactness of the Mayer-Vietoris sequence that
\begin{align*}
\tilde{\beta}_i(A_1\sqcup B\sqcup A_2) = \tilde{\beta}_i(A_1\sqcup B) + \tilde{\beta}_i(A_2\sqcup B) -
\tilde{\beta}_i(B) & \mbox{ for } 0\leq i\leq k-2.
\end{align*}
Therefore, we can compute

\begin{align*}
\sigma_i^{\ast}(X; \FF) & = \frac{1}{m_1+m_2-m+1}\sum_{A_1, A_2, B}\frac{\tilde{\beta}_i(A_1\sqcup B) +
\tilde{\beta}_i(A_2\sqcup B) -\tilde{\beta}_i(B)}{\binom{m_1+m_2-m}{\#(A_1\sqcup A_2\sqcup B)}} \\
&=  \frac{1}{m_1+m_2-m+1}\left[\sum_{A_1, B}\tilde{\beta}_i(A_1\sqcup B) \sum_{A_2}
\frac{1}{\binom{m_1+m_2-m}{\#(A_1\sqcup B) + \#(A_2)}} \right. \\
& \left. + \sum_{A_2, B}\tilde{\beta}_i(A_2\sqcup B) \sum_{A_1} \frac{1}{\binom{m_1+m_2-m}{\#(A_2\sqcup B)
+ \#(A_1)}} - \sum_{B}\tilde{\beta}_i(B) \sum_{A_1, A_2}
\frac{1}{\binom{m_1+m_2-m}{\#(B) + \#(A_1\sqcup A_2)}}\right]\\
&=  \frac{1}{m_1+m_2-m+1}\left[\sum_{k=0}^{m_1}\sum_{\stackrel{A_1, B}{\#(A_1\sqcup B)=k}}
\hspace{-6mm}\tilde{\beta}_i(A_1 \sqcup B) \sum_{\ell = 0}^{m_2-m}
\frac{\binom{m_2-m}{\ell}}{\binom{m_1+m_2-m}{\ell +k}} \right. \\
& \hspace{45mm} + \sum_{k=0}^{m_2}\sum_{\stackrel{A_2, B}{\#(A_2\sqcup B)=k}}
\hspace{-6mm}\tilde{\beta}_i(A_2 \sqcup B)
\sum_{\ell = 0}^{m_1-m} \frac{\binom{m_1-m}{\ell}}{\binom{m_1+m_2-m}{\ell +k}} \\
 & \hspace{50mm} \left. - \sum_{k=0}^{m}\sum_{\stackrel{B}{\#(B)=k}} \hspace{-3mm}\tilde{\beta}_i(B)
 \sum_{\ell = 0}^{m_1+m_2-2m} \frac{\binom{m_1+m_2-2m}{\ell}}{\binom{m_1+m_2-m}{\ell +k}}\right]\\
&=  \frac{1}{m_1+1}\sum_{k=0}^{m_1}\sum_{\stackrel{A_1, B}{\#(A_1\sqcup B)=k}} \hspace{-5mm}
\frac{\tilde{\beta}_i(A_1\sqcup B)}{\binom{m_1}{k}}  + \frac{1}{m_2+1}\sum_{k=0}^{m_2}\sum_{\stackrel{A_2,
B}{\#(A_2\sqcup B)=k}} \hspace{-5mm}
\frac{\tilde{\beta}_i(A_2 \sqcup B)}{\binom{m_2}{k}} \\
 & \hspace{60mm} - \frac{1}{m+1}\sum_{k=0}^{m}\sum_{\stackrel{B}{\#(B)=k}} \hspace{-3mm}
 \frac{\tilde{\beta}_i(B)}{\binom{m}{k}}\\
& =  \sigma_i^{\ast}(X_1; \FF) + \sigma_i^{\ast}(X_2; \FF) - \sigma_i^{\ast}(Y; \FF).
 \end{align*}
\end{proof}

In the penultimate step of the above computation, we have used the following well known identity to compute the
inner sums. For any three non-negative integers $p, q, r$, we have
\begin{align}
\sum_{i=0}^p \frac{\binom{p}{i}}{\binom{p+q+r}{i+r}} = \frac{p+q+r+1}{q+r+1}\times\frac{1}{\binom{q+r}{r}}.
\nonumber
\end{align}

As a particular case of Theorem \ref{theorem:5.7}, we have a formula for the sigma-star vector of a connected
sum. Since $\sigma_0^{\ast}(X; \FF)$ is independent of the field $\FF$, we denote it by $\sigma_0^{\ast}(X)$ in
the following.

\begin{corollary} \label{coro:5.8}
For any two triangulated $d$-spheres $S_1$, $S_2$ of dimension $d\geq 2$ and any field $\FF$, we have
\begin{enumerate}[{\rm (a)}]
\item $\sigma_0^{\ast}(S_1\# S_2) = \sigma_0^{\ast}(S_1) + \sigma_0^{\ast}(S_2) + \frac{1}{d+2}$, and

\item $\sigma_i^{\ast}(S_1\# S_2; \FF) = \sigma_i^{\ast}(S_1; \FF) + \sigma_i^{\ast}(S_2; \FF)$ for $1\leq i\leq
d-2$.
\end{enumerate}
\end{corollary}

\begin{proof}
Let $S_1\cap S_2 = \overline{\alpha}$, where $\alpha$ is a $d$-face. Then $X : = S_1\# S_2$, $X_1 := S_1\setminus
\{\alpha\}$, $X_2 := S_2\setminus\{\alpha\}$, $Y = \partial\overline{\alpha} = S^{\hspace{.15mm}d-1}_{d +1}$
satisfy the hypothesis of Theorem \ref{theorem:5.7} (with $k=d$), and trivially
$\sigma_0^{\ast}(S^{\hspace{.15mm} d -1}_{d+1}) = -1/(d+2)$, $\sigma_i^{\ast}(S^{\hspace{.15mm}d-1}_{d+1}) = 0$
for $1\leq i\leq d-2$.
\end{proof}

\begin{notation} \label{notation:5.9}
{\rm For $k, \ell\geq 0$; $(k, \ell)\neq (0, 0)$, we denote by $kI^2_{12}\# \ell S^{\hspace{.15mm}2}_4$ a
triangulated $2$-sphere which can be written as a connected sum of $k+\ell$ triangulated 2-spheres, of which $k$
are copies of $I^2_{12}$ and the remaining $\ell$ are copies of $S^{\hspace{.15mm}2}_4$ $($in some order$)$. }
\end{notation}

\begin{corollary} \label{coro:5.10}
Let $k, \ell\geq 0$ and $(k, \ell) \neq (0,0)$. Then $\sigma_0^{\ast}(kI^2_{12}\# \ell S^{\hspace{.15mm}2}_4) =
\frac{617}{1716}k + \frac{1}{20}\ell - \frac{1}{4}$.
\end{corollary}

\begin{proof}
Trivially, $\sigma_0^{\ast}(S^{\hspace{.15mm}2}_{4}) = -1/5$. A computation shows that $\sigma_0^{\ast}
(I^{\hspace{.15mm}2}_{12}) = 47/429$. Thus the result holds when $k+\ell = 1$. The general result follows by an
induction on $k+\ell$, where the induction leap uses Corollary \ref{coro:5.8} (with $d=2$).
\end{proof}

Our next result lists a set of necessary conditions that a $\ZZ_2$-tight, triangulated 3-manifold must satisfy.
Note that the statement of the main result of \cite{BDSS} implies that the inequality in part (a) of Theorem
\ref{theorem:5.11} holds, more generally, for all triangulations of closed 3-manifolds. Equality holds in this
more general setting only if the triangulation is neighbourly and locally stacked. Therefore, Proposition
\ref{prop:BDTh2.24} implies that equality holds if and only if the triangulation is $\ZZ_2$-tight and stacked.
Still, we have included this inequality for the sake of completeness, and because its proof arises naturally in
the course of proving the rest of the theorem.

\begin{theorem} \label{theorem:5.11}
Let $M$ be a $\ZZ_2$-tight, closed, triangulated $3$-manifold. Then the parameters $n :=f_0(M)$ and $\beta_1 :=
\beta_1(M;\ZZ_2)$ must satisfy the following.
\begin{enumerate}[{\rm (a)}]
\item $(n-4)(n-5) \equiv 20\beta_1$ $($mod $776)$. Also, $(n-4)(n-5) \geq 20\beta_1$, with equality if and only
if $M$ is stacked.
\item  $429(n-4)(n-5) - 776n\lfloor\frac{n-4}{9}\rfloor \leq 8580 \beta_1$. Equality holds here if and only if
each vertex link of $M$ is a triangulated $2$-sphere of the form
$\lfloor\frac{n-4}{9}\rfloor I^2_{12}\# (n-4-9\lfloor\frac{n-4}{9}\rfloor)S^{\hspace{.15mm}2}_4$.
\end{enumerate}
\end{theorem}

\begin{proof}
By Theorem \ref{theorem:5.6}, for each $x\in V(M)$ there are numbers $k(x)$ and $\ell(x)$ such that
\begin{align*}
M_x = k(x)I^2_{12} \# \ell(x) S^{\hspace{0.15mm}2}_4.
\end{align*}
Since $M$ is neighbourly by Lemma \ref{lemma:2.2}, equating the number of vertices in the two sides, we get $n-1
= 9k(x) + 3 + \ell(x)$. Hence $\ell(x) = n-4-9k(x)$, and we have $0\leq k(x)\leq \lfloor\frac{n-4}{9}\rfloor$.
Equality holds in the lower bound if and only if $M_x$ is stacked, and equality holds in the upper bound if and
only if $M_x$ is the connected sum of $\lfloor\frac{n-4}{9}\rfloor$ copies of $I^2_{12}$ and $n-4-
9\lfloor\frac{n-4}{9}\rfloor$ copies of $S^{\hspace{0.15mm}2}_4$. Let us put $k := \sum_{x\in V(M)} k(x)$. Then
$0\leq k\leq n\lfloor\frac{n-4}{9}\rfloor$. The lower bound holds with equality if and only if $M$ is locally
stacked (hence, by Proposition \ref{prop:BDTh2.24}, if and only if $M$ is stacked) and the upper bound holds with
equality if and only if each vertex link of $M$ is the connected sum of $\lfloor\frac{n-4}{9}\rfloor$ copies of
$I^2_{12}$ and $n-4- 9\lfloor\frac{n-4}{9}\rfloor$ copies of $S^{\hspace{0.15mm}2}_4$. Then, by Corollary
\ref{coro:5.10},

\begin{align} \label{eq:1}
\sigma_0^{\ast}(M_x) & = \frac{617}{1716}k(x) + \frac{1}{20}\ell(x) - \frac{1}{4} 
= \frac{617}{1716}k(x) + \frac{1}{20}(n-4-9k(x)) - \frac{1}{4} \nonumber \\
& = \frac{1}{20}n-\frac{194}{2145}k(x) - \frac{9}{20}.
\end{align}
Now Lemma \ref{lemma:2.6} implies $\beta_1 = \mu_1(M) := 1 + \sum_x\sigma_0^{\ast}(M_x)$. Therefore, adding
\eqref{eq:1} over all $x\in V(M)$, we get
\begin{align} \label{eq:2}
429((n-4)(n-5)-20\beta_1)= 776k.
\end{align}
Since 776 is relatively prime to 429, the result follows from \eqref{eq:2} and the above discussion on the
bounds on $k$.
\end{proof}

\begin{corollary}[An upper bound theorem for $\ZZ_2$-tight, triangulated 3-manifolds] \label{coro:5.12}
Let $M$ be a $\ZZ_2$-tight, closed, triangulated $3$-manifold with $n :=f_0(M)$, $\beta_1 := \beta_1(M;\ZZ_2)$.
Then $(n- 4)(617n -3861) \leq 15444\beta_1$. Equality holds here if and only if $M$ is locally icosian.
\end{corollary}

\begin{proof}
Since $\lfloor\frac{n-4}{9} \rfloor \leq \frac{n-4}{9}$, Theorem \ref{theorem:5.11} (b) implies that $429(n-4)(n
-5) - 776n(n-4)/9 \leq 8580\beta_1$. This simplifies to the given inequality. Clearly, equality holds here if and
only if $n \equiv 4$ (mod 9) and equality holds in Theorem \ref{theorem:5.11} (b).
\end{proof}

\begin{corollary} \label{coro:5.13}
Let $M$ be a $\ZZ_2$-tight, closed, triangulated $3$-manifold. If $f_0(M) \leq 71$ or $\beta_1(M;\ZZ_2) \leq
188$, then  $M$ is stacked.
\end{corollary}

\begin{proof}
Suppose, if possible, $M$ is not stacked. Then, by Theorem \ref{theorem:5.11} (a), there is an integer $\ell \geq
1$ such that $ (n-4)(n- 5) = 776\ell +20\beta_1$. Hence Corollary \ref{coro:5.12} implies
\begin{align*}
15444(n-4)(n-5) \geq 15444\times 776\ell + 20(617n-3861)(n-4).
\end{align*}
This inequality simplifies to
\begin{align} \label{eq:3}
(n-2)^2 \geq 3861\ell +4.
\end{align}

\noindent {\em Case 1.} $n \leq 71$. Hence, by \eqref{eq:3}, $3861\ell + 4 \leq (71-2)^2 = 4761$. Thus, $\ell=1$.
Therefore, $(n-4)(n-5) = 776 +20\beta_1$ and $(n-2)^2 \geq 3861+4 > 62^2$. Thus, $n\geq 65$. But, $(n-4)(n-5)
\equiv 776 \equiv -4$ (mod 20). Hence $n \equiv 12$ or 17 (mod 20). But, $n\geq 65$. So $n\geq 72$, a
contradiction. So the result is true in this case.

\noindent {\em Case 2.} $\beta_1(M;\ZZ_2) \leq 188$. Then, by Corollary \ref{coro:5.12}, $n\leq 73$. Hence
\eqref{eq:3} yields $\ell=1$. Thus $(n-4)(n-5) = 776 + 20 \beta_1$, and therefore $n$ is congruent to 12 or 17
(mod 20). Hence $n \leq 72$, and if $n=72$ then $\beta_1 = 189$. Therefore, $n\leq 71$. The result now follows by
Case 1.
\end{proof}

\begin{corollary} \label{coro:5.14}
Let $M$ be a closed, topological $3$-manifold admitting a tight triangulation. If $\beta_1 (M,\mathbb{Z}_2) <
189$, then $M$ is homeomorphic to one of the following manifolds
\begin{align*}
S^3, \, (S^2 \times S^1)^{\#k}, \, (\TPSS)^{\#k},
\end{align*}
where $k = 1, 12, 19, 21, 30, 63, 78, 82, 99, 154, 177$ or $183$.
\end{corollary}

\begin{proof}
Let $X$ be an $n$-vertex, tight triangulation of $M$. Then, by Corollary \ref{coro:5.13}, Theorem
\ref{theorem:5.11} (a) and Proposition \ref{prop:Walkup_class}, $M$ is homeomorphic to $S^3$, $(S^2 \times
S^1)^{\#k}$ or $(\TPSS)^{\#k}$, where $(n-4)(n-5) = 20k$. The result follows from this.
\end{proof}

The following table gives a list of small values for the parameters $(n, \beta_1)$ of a locally icosian
$\ZZ_2$-tight, closed,  triangulated 3-manifold. Indeed, the number $n$ of vertices in any such triangulated
3-manifold must be congruent modulo 15444 to one of the eight values of $n$ listed in this table.

\begin{table}[ht]
\centering {\small
\begin{tabular}{|c|c|c|c|c|c|c|c|c|}
\hline
&&&&&&&&\\[-3mm]
$n$ & 1408 & 3865 & 5269 & 8320 &9724&12181&13585& 15448\\[1mm]
\hline
&&&&&&&&\\[-3mm]
$\beta_1$ & 78625 & 595186 & 1106970 &2762081 &3773610&5922778&7367441&9527555 \\[1mm]
\hline
\end{tabular}
} \caption{Small feasible parameters for locally icosian, $\ZZ_2$-tight 3-manifolds} \label{tbl:icsian.3mfd}
\end{table}


The following tables list the small values for parameters $(n, \beta_1)$ satisfying the conclusion of Theorem
\ref{theorem:5.11}. Table \ref{tbl:Z2tight.3mfd} for $(n, \beta_1)$ with strict inequality in part (a) of the
theorem. Table \ref{tbl:Z2tight.3mfd.c=} for $(n, \beta_1)$ with equality in part (b) of the theorem.
\begin{table}[ht]
\centering {\small
\begin{tabular}{|c|c|c|c|c|c|c|c|c|c|c|c|c|c|c|c|}
\hline
&&&&&&&&&&&&&&&\\[-3mm]
$n$ & 72 & 77 & 92 & 96 &97&101&108& 112&113&116&117&121&128&132&133\\[1mm]
\hline
&&&&&&&&&&&&&&&\\[-3mm]
$\beta_1$ &189&224&344&341&389&388&458&539&511&544&594&601&685&774&748\\[1mm]
\hline
\end{tabular}
} \caption{Small feasible parameters for non-stacked, $\ZZ_2$-tight 3-manifolds} \label{tbl:Z2tight.3mfd}
\end{table}
\begin{table}[hb]
\centering {\small
\begin{tabular}{|c|c|c|c|c|c|c|c|c|c|c|c|c|}
\hline
&&&&&&&&&&\\[-3mm]
$n$ &825&1296&1408&1760&1881&1989&2145&2580&3168&3276\\[1mm]
\hline
&&&&&&&&&&\\[-3mm]
$\beta_1$ &26871&66637&78625&123049&140677&157336&183109&264924&399817&427582\\[1mm]
\hline
\end{tabular}
} \caption{Small feasible parameters for $\ZZ_2$-tight 3-manifolds (with equality in Th. \ref{theorem:5.11} (b))}
\label{tbl:Z2tight.3mfd.c=}
\end{table}

\section{Examples}

In this section, we present some examples which help to put the results of this paper in proper perspective.

The first example shows that a strongly minimal, stacked, closed, triangulated $3$-manifold need not be tight. So
the converse of Proposition \ref{prop:sminimal} is not true.

\begin{example}[Walkup \cite{Wa}] \label{exam:Walkup}
{\rm Let $\mathcal{J}$ be the pure 4-dimensional simplicial complex with vertex set $\ZZ_{10} = \ZZ/10\ZZ$ and an
automorphism $i\mapsto i+1$ (mod 10). Modulo this automorphism, there is only one representative facet (maximal
face) of $\mathcal{J}$, namely $12345$. The face vector of $\mathcal{J}$ is $(10, 40, 60, 40, 10)$. Each vertex
link of $\mathcal{J}$ is an 8-vertex stacked 3-ball and hence $\mathcal{J}$ is a locally stacked, triangulated
4-manifold with boundary. The boundary $\mathcal{K} = \partial\mathcal{J}$ has face vector $(10, 40, 60, 30)$ and
triangulates $S^{\hspace{.15mm}2} \times S^1$. Since each 2-simplex of $\mathcal{J}$ is in the boundary,
$\mathcal{J}$ is stacked and hence $\mathcal{K}$ is a stacked, closed, triangulated 3-manifold. Since
$\mathcal{K}$ is not neighbourly, it is not tight with respect to any field.

Let $Y$ be a triangulation of $S^{\hspace{.15mm}2} \times S^1$. Since the only closed, triangulated $3$-manifolds
with at most $9$ vertices are $S^{\hspace{.15mm}3}$ and $\TPSS$ (see \cite{Wa}), it follows that $f_0(Y) \geq
10$. Then, by \cite[Theorem 5.2]{NS}, $f_1(Y) \geq 4f_0(Y) \geq 40$. Since $f_0(Y)- f_1(Y) +f_2(Y) -f_3(Y) = 0$
and $2f_2(Y) = 4f_3(Y)$, it follows that $f_3(Y) = f_1(Y) -f_0(Y) \geq 4f_0(Y) - f_0(Y) = 3f_0(Y)\geq 30$. Hence
$f_2(Y) = 2f_3(Y) \geq 60$. Thus, $f_i(Y) \geq f_i(\mathcal{K})$ for $0\leq i\leq 3$. Therefore, $\mathcal{K}$ is
strongly minimal.}
\end{example}

The following example shows that a neighbourly, locally stacked, closed, triangulated 3-manifold need not be
stacked. Equivalently, it need not be tight. Thus, the hypothesis `stacked' in Theorem \ref{theorem:4.8} can not
be relaxed to `locally stacked'.

\begin{example}[Lutz \cite{LuThesis}] \label{exam:Lutz}
{\rm Let $\mathcal{L}$ be the pure 3-dimensional simplicial complex with vertex set $\ZZ_{10} = \ZZ/10\ZZ$ and an
automorphism $i\mapsto i+1$ (mod 10). Modulo this automorphism, a set of representative facets of $\mathcal{L}$
is:
\begin{align}
1236, \, 1237, \, 1257, \, 1368. \nonumber
\end{align}
Then, each vertex link is a 9-vertex stacked 2-sphere and hence $\mathcal{L}$ is a locally stacked, neighbourly,
closed, triangulated 3-manifold. It triangulates $S^2 \times S^1$ and hence $\beta_1(\mathcal{L}; \FF) = 1$ for
any field $\FF$. Therefore, by Proposition \ref{prop:BDTh2.24}, $\mathcal{L}$ is not stacked  and not tight with
respect to any field. Since there are non-neighbourly 10-vertex triangulations of $S^2 \times S^1$ (cf.
$\mathcal{K}$ in Example \ref{exam:Walkup}), $\mathcal{L}$ is not strongly minimal. By Propositions
\ref{prop:sminimal} and \ref{prop:BDTh2.24}, this also shows that $\mathcal{L}$ is not stacked and is not
$\FF$-tight for any field $\FF$. Since $\mathcal{L}$ is not stacked, $\mathcal{L}$ is not in $\mathcal H^{4}$ by
Proposition \ref{prop:Walkup_class}.}
\end{example}

The following example shows that an arbitrary $\FF$-tight simplicial complex need not be strongly minimal (we do
not know if it must be minimal in the sense of having the fewest number of vertices among all triangulations of
its geometric carrier).

\begin{example} \label{exam:nsmtight}
{\rm Consider the neighbourly $2$-dimensional simplicial complex $X$ on the vertex set $\{1, 2, 3, 4\}$ whose
maximal faces are $123$, $234$ and $14$. It is easy to see that $X$ is $\FF$-tight for any field $\FF$. Observe
that $|X|$ is also triangulated by the simplicial complex $Z$ (on the same vertex set) whose maximal faces are
$123$, $14$, $24$. Thus, $X$ is not strongly minimal. Therefore, Conjecture \ref{conj:KL} is not true for
arbitrary simplicial complexes. }
\end{example}

Recall that a $d$-dimensional simplicial complex is a {\em pseudo-manifold} if (i) each maximal face is
$d$-dimensional, (ii) each $(d-1)$-face is in at most two $d$-faces, and (iii) for any two $d$-faces $\alpha$ and
$\beta$, there exists a sequence $\alpha=\alpha_1, \dots, \alpha_m=\beta$ of $d$-faces such that $\alpha_i
\cap\alpha_{i+1}$ is a $(d-1)$-face for $1\leq i\leq m-1$. We now extend the definitions of stackedness and
locally stackedness to pseudo-manifolds as follows.

\begin{definition} \label{def:st-pmfd}
{\rm A pseudo-manifold $Q$ of dimension $d+1$ is said to be {\em stacked} if all its faces of codimension (at
least) two are in the boundary $\partial Q$. A pseudo-manifold $P$ without boundary of dimension $d$ is said to
be {\em stacked} if there is a stacked pseudo-manifold $Q$ of dimension $d+1$ such that $P = \partial Q$. A
pseudo-manifold is said to be {\em locally stacked} if all its vertex links are stacked pseudo-manifolds (with or
without boundaries).}
\end{definition}

\begin{example}[Emch \cite{Em}] \label{exam:3-pmfd}
{\rm Consider the $3$-dimensional pseudo-manifold $\mathcal{P}$ with vertex set $\{1, \dots, 8\}$ and the
facet-transitive automorphism group ${\rm PGL}(2,7) = \langle (12345678), (132645),$ $(16)(23)(45)(78)\rangle$.
Modulo this group, a facet representative is $1235$. The link of each vertex is isomorphic to the 7-vertex torus
$T^2_7$. Its face vector is $(8, 28, 56, 28)$.

Since $\mathcal{P}$ is a $3$-neighbourly pseudo-manifold with Euler characteristic $8$, it follows that its
integral homologies are torsion free and its vector of Betti numbers is $(1,0,8,1)$. Using its doubly transitive
automorphism group, one may calculate easily that the sigma vector of $T^2_7$ (with respect to any field) is
$(-1,8,1)$. Hence the mu-vector of $\mathcal{P}$ (with respect to any field) also equals $(1,0,8,1)$. Therefore,
\cite[Theorem 1.6]{BaMu} implies that $\mathcal{P}$ is tight with respect to all fields.

Now, if $Q$ is a pseudo-manifold of dimension 4 such that $\partial Q = \mathcal{P}$ and ${\rm skel}_2(Q) = {\rm
skel}_2(\mathcal{P})$ then for any vertex $v$ of $Q$, we have $\partial(Q_v) = \mathcal{P}_v$ and ${\rm
skel}_1(Q_v) = {\rm skel}_1(\mathcal{P}_v)$. Thus, $Q_v$ is a stacked, triangulated 3-manifold whose boundary is
the 7-vertex torus $P_v$. But it is easy to see that the 7-vertex torus bounds exactly three (distinct but
isomorphic) stacked, triangulated 3-manifolds, each with seven 3-faces. This implies that $f_4(v\ast Q_v) =
f_3(Q_v) = 7$. Then, $f_4(Q) = (8\times 7)/5$, which is not possible. Therefore, $\mathcal{P}$ is not stacked.
Thus, Theorems \ref{theorem:4.8} and \ref{theorem:5.1} are not true for $3$-dimensional pseudo-manifolds.

}
\end{example}

The next example disproves a putative generalization of Theorem \ref{theorem:3.5} in which the induced cycle in
the hypothesis is replaced by arbitrary manifolds without boundary.

\begin{example} \label{exam:end}
{\rm Consider the $6$-vertex triangulation $\mathbb{RP}^2_6$ of $\mathbb{RP}^2$ whose facets are $123$, $134$,
$145$, $156$, $126$, $235$, $245$, $246$, $346$, $356$. Let $X = \overline{123456} \cup (7\ast \mathbb{RP}^2_6)$.
So $X$ is the union of a standard 5-ball and a cone over $\mathbb{RP}^2_6$. The simplicial complex $X$ is
isomorphic to an induced subcomplex of the 13-vertex triangulation $M$ of $SU(3)/SO(3)$ obtained by Lutz in
\cite{LuThesis}. (Indeed, for each vertex $v$ of $M$, $M_v$ contains two induced $\mathbb{RP}^2_6$'s, say
$M_v[A]$ and $M_v[B]$, where $V(M_v) = A\sqcup B$. Both the induced subcomplexes $M[\{v\}\cup A]$ and $M[\{v\}
\cup B]$ are isomorphic to $X$.) Thus, $M$ is $\ZZ_2$-tight \cite{KL}, $M_v[A] = \mathbb{RP}^2_6$ is a surface
and yet the induced subcomplex $X= M[\{v\}\cup A]$ is not a pseudo-manifold. }
\end{example}

\medskip

\noindent {\bf Acknowledgements:} The second and third authors are supported by  DIICCSRTE, Australia (project
AISRF06660) and DST, India (DST/INT/AUS/P-56/2013(G)) under the Australia-India Strategic Research Fund. The
second author is also supported by the UGC Centre for Advanced Studies (F.510/6/CAS/2011 (SAP-I)). 

{\small

}

\end{document}